\documentclass{article}
\usepackage{graphicx}
\usepackage{graphicx}
\usepackage{amssymb}
\usepackage{enumerate}
\usepackage{amsthm}
\usepackage{textcomp}
\usepackage{graphicx,epsfig}
\usepackage{amsfonts,amssymb,enumerate}
\usepackage{latexsym,amsfonts,amssymb,amsmath,amscd,eufrak}
\usepackage{fancybox,shadow,color,graphicx,psfrag,float}
\usepackage{pstricks,pst-node,pst-text,pst-3d,pst-plot}
\usepackage{tikz,pgffor}
\newtheorem{theorem}{Theorem}[section] 
\newtheorem{cor}[theorem]{Corollary}
\newtheorem{proposition}[theorem]{Proposition}
\theoremstyle{definition}
\newtheorem{example}{Example}

\theoremstyle{remark}
\newtheorem{rem}[theorem]{Remark}
\numberwithin{equation}{section}

\newcommand{\ff}{\mathbb{F}}

\newcommand{\m}{\mathcal}
\renewcommand{\P}{\mathbb{P}}
\renewcommand{\deg}[1]{deg\,#1}

\newcommand{\diff}{\operatorname{Diff}}

\newcommand{\supp}{\operatorname{supp}}

\begin{document}

\title{Asymptotically bad towers of function fields}%
\author{M. Chara\footnote{IMAL (UNL-CONICET), Colectora Ruta Nac. N° 168 km. 472, Paraje El Pozo, (3000) Santa Fe, Argentina. {\it e-mail: mchara@santafe-conicet.gov.ar}}\quad R. Toledano\footnote{IMAL (UNL-CONICET) and FIQ (UNL), Departamento de Matem\'atica, Facultad de Ingenier\'ia Qu\'imica, Stgo. del Estero 2829, (3000) Santa Fe, Argentina. {\it e-mail: rtoledano@santafe-conicet.gov.ar}}
}

\maketitle

\date{}%
\begin{abstract}
In this paper we study general conditions to prove the infiniteness of the genus of certain towers of function fields over a perfect field. We show that many known examples of towers with infinite genus are particular cases of these conditions.



\end{abstract}

\section{Introduction} 
\label{intro}

The aim of this paper is to study, on the one hand, general sufficient conditions to prove the infiniteness of the genus of certain towers of function fields over a perfect field  and, on the other hand, to give some criteria to construct concrete examples of towers with infinite genus. We will work in the general setting of towers $\m{F}=(F_0, F_1, \ldots)$ of function fields $F_i$ over a perfect field $K$  with some additional properties (see Section \ref{notanddef} for precise definitions). Roughly speaking the type of towers we will consider in this paper are towers in the sense of Garcia and Stichtenoth (see Chapter $7$ of \cite{Stichbook09}) without asking the condition that $g(F_i)$, the genus of $F_i$, tends to infinity as $i$ tends to infinity. The reason in doing so is that there is no need of this condition when proving that a given tower has infinite genus. Once this is proved we have that  $g(F_i)$ must tend to infinity as $i$ tends to infinity so that the considered tower is, in fact, a tower in the sense of Garcia and Stichtenoth which is, additionally, asymptotically bad.

In the case of recursive towers (i.e. each extension $F_{i+1}/F_i$ is defined by using some of the roots of a bivariate polynomial with coefficients in $K$) there is a simple sufficient condition for bad asymptotic behavior, namely  that the recursive tower is skew in the sense that the degree of the polynomial defining the tower is not the same in each variable (see \cite{GS07} for details). Unfortunately this condition is not necessary because there are examples of non skew asymptotically bad towers. Interestingly, when a tower is non skew, it seems equally hard to prove that the tower is asymptotically good or asymptotically bad.


The organization of the paper is as follows. In Section \ref{notanddef} we give the basic definitions and we establish the notation to be used throughout the paper. In Section \ref{bad} we prove our main results in the general setting of sequences of function fields over a perfect field $K$ and in Section \ref{examples} we give  examples of
asymptotically bad towers showing on the one hand that many known examples are particular cases of our general results and, on the other hand, that all can be obtained in a unified way. The key in this general and unified treatment of the bad asymptotic behavior of towers is the existence of a particular divisor which is far from being obvious in some of the given examples.

\section{Notation and Definitions}\label{notanddef}
In this work we shall be concerned with \emph{ towers} which mean that $\m{F}=(F_0, F_1, \ldots)$ is an infinite sequence of function fields over $K$ where for each index $i\geq 0$ the field $F_i$ is a proper subfield of $F_{i+1}$, the field extension $F_{i+1}/F_i$ is finite and separable and $K$ is the full field of constants of each field $F_i$ (i.e. $K$ is algebraically closed in each $F_i$). If the genus $g(F_i)\rightarrow \infty$ as $i\rightarrow \infty$ we shall say that $\m{F}$ is a {\em tower in the sense of Garcia and Stichtenoth}.

Following \cite{Stichbook09}  (see also \cite{GS07}), one way of constructing towers of function fields over $K$ is by giving a bivariate polynomial
\[H\in K[X,Y]\,,\]
 and a transcendental element $x_0$ over $K$. In this situation a tower $\mathcal{F}=(F_0, F_1, \ldots)$ of function fields over $K$ is defined as
 \begin{enumerate}[(i)]
 \item $F_0=K(x_0)$, and
 \item $F_{i+1}=F_i(x_{i+1})$ where $H(x_i,x_{i+1})=0$ for $i\geq 0$.
 \end{enumerate}
 A suitable choice of the bivariate polynomial $H$ must be made in order to have towers. When the choice of $H$ satisfies all the required conditions we shall say that the tower $\mathcal{F}$ constructed in this way is a {\em recursive tower} of function fields over $K$. Note that for a recursive tower $\mathcal{F}=(F_0, F_1, \ldots)$ of function fields over $K$ we have that
  \[F_i=K(x_0,\ldots,x_i)\qquad \text{for }i\geq 0,\]
where $\{x_i\}_{i=0}^{\infty}$ is a sequence of transcendental elements over $K$.

Associated to a recursive tower  $\mathcal{F}=(F_0, F_1, \ldots)$ of function fields $F_i$  over $K$ we have the so called {\em basic function field} $K(x,y)$ where $x$ is transcendental over $K$ and $H(x,y)=0$.

For the sake of simplicity we shall say from now on that $H$ defines the tower $\m{F}$ or, equivalently, that  tower $\m{F}$ is recursively defined by the equation $H(x,y)=0$.

A tower $\m{F}=(F_0, F_1, \ldots)$ of function fields over a perfect field $K$ of positive characteristic is called \emph{tame} if the ramification index $e(Q|P)$ of any place $Q$ of $F_{i+1}$ lying above a place $P$ of $F_i$ is relatively prime to the characteristic of $K$ for all $i\geq 0$. Otherwise the tower $\m{F}$ is called \emph{wild}.

 The set of places of a function field $F$ over $K$ will be denoted by $\P(F)$.

The following definitions are important when dealing with the asymptotic behavior of a tower. Let $\mathcal{F}=(F_0, F_1, \ldots)$ be a  tower of function fields
over a finite field $\ff_q$ with $q$ elements.
The {\em splitting rate} $\nu(\m{F})$ and the {\em genus} $\gamma(\m{F})$ of $\m{F}$ over $F_0$ are defined, respectively, as
$$\nu(\m{F})\colon=\lim_{i\rightarrow \infty}\frac{N(F_i)}{[F_i:F_0]}\,, \qquad\gamma(\m{F})\colon=\lim_{i\rightarrow \infty}\frac{g(F_i)}{[F_i:F_0]}\,.$$
If $g(F_i)\geq 2$
for $i\geq i_0\geq 0,$ the {\em limit} $\lambda(\m{F})$ of $\m{F}$ is defined as $$\lambda(\m{F})\colon=\lim_{i\rightarrow \infty}\frac{N(F_i)}{g(F_i)}\,.$$
It can be seen that all the above limits exist and that $\lambda(\m{F})\geq 0$ (see \cite[Chapter 7]{Stichbook09}).

We shall say that a tower $\mathcal{F}=(F_0, F_1, \ldots)$ of function fields over $\ff_q$ is {\em asymptotically good} if $\nu(\m{F})>0$ and $\gamma(\m{F})<\infty$. If either $\nu(\m{F})=0$ or $\gamma(\m{F})=\infty$ we shall say that $\m{F}$ is {\em asymptotically bad}.

From the well-known Hurwitz genus formula
(see \cite[Theorem 3.4.13]{Stichbook09}) we see that the condition $g(F_i)\geq 2$ for $i\geq i_0$ in the definition of $\lambda(\m{F})$ implies that $g(F_i)\rightarrow \infty$ as $i\rightarrow \infty$.
Hence, when  we speak of the limit of a tower of function fields we are actually speaking of the limit of a tower in the sense of Garcia and Stichtenoth (see \cite[Section 7.2]{Stichbook09}).

It is easy to check that in the case of a tower $\m{F}$ we have that $\m{F}$ is asymptotically good if and only if $\lambda(\m{F})>0$. Therefore a tower $\m{F}$ is  asymptotically bad if and only if $\lambda(\m{F})=0$.

\section{Bad towers}\label{bad}

As was mentioned in the introduction, a simple and useful condition implying that $H \in \ff_q[x,y]$ does not give rise to an asymptotically good recursive tower  $\m{F}$ of function fields over $\ff_q$ is that $\deg_xH\neq\deg_yH$. With this situation in mind
we shall say that a recursive tower $\m{F}=(F_0, F_1, \ldots)$ of function fields over a perfect field $K$ defined by a polynomial $H \in K[x,y]$ is {\em non skew} if $\deg_xH=\deg_yH$. In the skew case (i.e. $\deg_xH\neq \deg_yH$) we might have that $[F_{i+1}:F_i]\geq 2$ for all $i\geq 0$ and
even that $g(F_i)\rightarrow \infty$ as $i\rightarrow\infty$ but, nevertheless, $\m{F}$ will be asymptotically bad. What happens is that if $\deg_yH>\deg_xH$ then
the splitting rate  $\nu(\m{F})$ is zero (this situation makes sense in the case $K=\ff_q)$ and if  $\deg_xH>\deg_yH$ the genus $\gamma(\m{F})$ is infinite (see \cite{GS07} for details). Therefore the search for conditions for bad asymptotic behavior must be focused on non skew towers.

Since nothing else on the polynomial $H$ seems to help to decide if a tower defined by $H$ is asymptotically bad, the natural step in the search for conditions for bad asymptotic behavior of a tower $\m{F}=(F_0, F_1, \ldots)$ is to look for upper bounds for $N(F_i)$ or lower bounds for $g(F_i)$ with the hope of proving that the given tower has either zero splitting rate or infinite genus. In the present work we will focus in the latter case. The starting point in all the known results in this direction is the following proposition. From now on $K$ will denote a perfect field and we recall that $K$ is assumed to be the full  field of constants of each function field $F_i$ of any given tower $\m{F}$ over $K$.

\begin{proposition}\label{propgenusdiff}
Let $\m{F}=(F_0, F_1, \ldots)$ be a tower of function fields over $K$. Suppose that there is a subtower $\m{F}'=(F_{s_1}, F_{s_2}, \ldots)$
such that the series
\begin{equation}\label{divergentseries}
\sum_{i=1}^\infty
\frac{\deg{\diff(F_{s_i}/F_{s_{i-1}})}}{[F_{s_i}:F_0]}\,,
\end{equation}
 is divergent. Then $\gamma(\m{F})=\infty$ which implies that $\m{F}$ is an asymptotically bad tower of function fields over $K$. Reciprocally, if $\gamma(\m{F})=\infty$ then the series \eqref{divergentseries} is divergent for any subtower $(F_{s_1}, F_{s_2}, \ldots)$ of $\m{F}$. In particular, if $\m{F}$ is an asymptotically bad tower of function fields over $\ff_q$ and $\nu(\m{F})>0$ then the series \eqref{divergentseries} is divergent for any subtower $(F_{s_1}, F_{s_2}, \ldots)$.
\end{proposition}

\begin{proof}

The proposition follows easily from the fact that the genus $\gamma(\m{F})$ of the tower $\m{F}$ can be written as \begin{equation}\label{ecu1paper3}\gamma(\m{F})=
\lim_{i \rightarrow \infty}\frac {g(F_{s_i})}{[F_{s_i}:F_0]}=g(F_0)-1+\frac 1 2 \sum_{i=1}^\infty
\frac{\deg{\diff(F_{s_{i+1}}/F_{s_i})}}{[F_{s_{i+1}}:F_0]}.\end{equation} This can be easily proved by induction. In fact, using Hurwitz genus formula (see \cite[Theorem 3.4.13]{Stichbook09}) for the extension $F_{s_1}/F_0$ we have that
 $$2(g(F_{s_1})-1)=[F_{s_1}:F_0](2g(F_0)-2)+\deg{\diff(F_{s_1}/F_0)}$$ and hence
$$\frac{g(F_{s_1})-1}{[F_{s_1}:F_0]}=g(F_0)-1+\frac{1}{2} \frac{\deg{\diff(F_{s_1}/F_0)}}{[F_{s_1}:F_0]}.$$
Now suppose that
$$\frac{g(F_{s_i})-1}{[F_{s_i}:F_0]}=g(F_0)-1+\frac{1}{2} \sum_{j=1}^{i}\frac{\deg{\diff(F_{r_i}/F_{r_{j-1}})}}{[F_{r_i}:F_0]},$$ where
$F_{r_0}$ denote $F_0$. Then by Hurwitz genus formula for the extension $F_{r_{i+1}}/F_{s_i}$ we have
$$2(g(F_{r_{i+1}})-1)=[F_{r_{i+1}}:F_{s_i}](2g(F_{s_i})-2)+\deg{\diff(F_{r_{i+1}}/F_{s_i})}$$ and therefore
\begin{align*}
  \frac{g(F_{r_{i+1}})-1}{[F_{r_{i+1}}:F_0]}&=\frac{g(F_{s_i})-1}{[F_{s_i}:F_0]}+\frac{1}{2}\frac{\deg{\diff(F_{r_{i+1}}/F_{s_i})}}{[F_{r_{i+1}}:F_0]}\\
  &=g(F_0)-1+\frac{1}{2} \sum_{j=1}^i\frac{\deg{\diff(F_{r_i}/F_{r_{j-1}})}}{[F_{r_i}:F_0]}+\frac{1}{2}\frac{\deg{\diff(F_{r_{i+1}}/F_{s_i})}}{[F_{r_{i+1}}:F_0]}\\
  &=g(F_0)-1+\frac{1}{2} \sum_{j=1}^{i+1}\frac{\deg{\diff(F_{s_i}/F_{r_{j-1}})}}{[F_{s_i}:F_0]}.
\end{align*}

Since the limit ${g(F_n)}/{[F_n:F_0]}$ always exists and $[F_{r_{i+1}}:F_0]\rightarrow \infty$ as $i\rightarrow \infty$, we have that
\begin{align*}
\gamma(\m{F})&=\lim_{n \rightarrow \infty}\frac {g(F_n)}{[F_n:F_0]}\\
&=\lim_{n \rightarrow \infty}\frac {g(F_n)-1}{[F_n:F_0]}\\
&=\lim_{i \rightarrow \infty}\frac {g(F_{s_i})-1}{[F_{s_i}:F_0]}\\
&=g(F_0)-1+\frac{1}{2}\lim_{i \rightarrow \infty} \sum_{j=1}^{i+1}\frac{\deg{\diff(F_{s_i}/F_{r_{j-1}})}}{[F_{s_i}:F_0]}\\
&=g(F_0)-1+\frac 1 2 \sum_{j=1}^\infty
\frac{\deg{\diff(F_{s_i}/F_{r_{j-1}})}}{[F_{s_i}:F_0]},
\end{align*}
which is the desired result.
Finally, if $\gamma(\m{F})=\infty$ the reciprocal result follows easily from \eqref{ecu1paper3}.
\end{proof}

Proposition~\ref{propgenusdiff} was used in \cite{GS96},\cite{MaWu05}, \cite{BGS04} and \cite{BeelenGS05} where the authors give different conditions to construct non skew asymptotically bad towers. However not stated explicitly in the just mentioned articles, the infiniteness of the genus of the given towers is deduced from the following remark.

\begin{rem}\label{remarkdivisor}
Suppose that  $\m{F}=(F_0, F_1, \ldots)$ is a tower of function fields over  $K$ and that there exist positive
functions  $c_1(t)$ and $c_2(t)$, defined for $t\geq 0$, and  a divisor  $B_i\in \m{D}(F_i)$ such that for each $i\geq 1$
 \begin{enumerate}[(a)]
\item $\deg B_i\geq c_1(i)[F_i:F_0]$ and\label{thm3.2-a}
 \item $\sum\limits_{P\in supp(B_i)}\sum\limits_{Q|P}d(Q|P)\deg Q\geq c_2(i)[F_{{i+1}}\colon F_i]\deg{B_i}\,,$ \label{thm3.2-b}
 \end{enumerate}
where the inner sum runs over all places $Q$ of $F_{i+1}$ lying above $P$, then it is easy to see that if the series
\begin{equation}\label{thm3.2-c}
  \sum_{i=1}^{\infty}c_1(i)c_2(i)
\end{equation} is divergent then $\gamma(\m{F})=\infty$.
\end{rem}

\begin{cor}\label{coroTeoclasificacion1}
  With the same hypotheses as in Remark~\ref{remarkdivisor}, if in addition $\m{F}=(F_0, F_1, \ldots)$ is non skew and  recursively defined by the equation $H(x,y)=0$ such that $H(x,y)$, as a polynomial with coefficients in $K(y)$, is irreducible in $K(y)[x]$
  then condition \eqref{thm3.2-a}
  can be replaced by the following
\begin{enumerate}[(a')]
\item $\deg{B_{j}}\geq c_1(j)\cdot \deg(b(x_{j}))^{j}$ where $b\in K(T)$ is a rational function and $(b(x_{j}))^{j}$ denotes either the pole divisor or the zero divisor of $b(x_{j})$ in $F_{j}$,\label{thm3.2-a'}
\end{enumerate}
and the same result hold, i.e., $\gamma(\m{F})=\infty$.
\end{cor}

\begin{rem}\label{remteoclasi}
It is easy to see that Corollary~\ref{coroTeoclasificacion1} also hold if the conditions are expressed in terms of a subtower $\m{F}'=(F_{s_1}, F_{s_2}, \ldots)$ of $\m{F}$.
\end{rem}

  In the wild case the infiniteness of the genus of a tower $\m{F}=(F_0, F_1, \ldots)$ over $K$ can be proved by showing that for infinitely many indices $i\geq 1$ there is a place $Q$ of $F_{i+1}$ such that  the different exponent $d(Q|P)$ satisfies
\begin{equation}\label{diffexplowerbound}
d(Q|P)\geq c_i[F_i:F_0]\,,
\end{equation}
where $P$ is the place of $F_i$ lying under $Q$ and $c_i>0$ is such that the series
\begin{equation}\label{seriesdiffexp}
\sum_{i=1}^{\infty}\frac{c_i}{[F_{i+1}:F_i]}\,,
\end{equation}
is divergent. In this case it is easy to check that the divisor \[B_i=[F_i:F_0]P\in \m{D}(F_i)\,,\] satisfies conditions \eqref{thm3.2-a} and \eqref{thm3.2-b} of Remark \ref{remarkdivisor} with $c_1(i)=c_i$ and $c_2(i)=[F_{i+1}:F_i]^{-1}$. An example of the situation just described was given in \cite{GS96} where the authors showed that if a tower $\m{F}=(F_0, F_1, \ldots)$ of function fields over $\ff_2$ is recursively generated by the equation
\[y^3+y=\frac{x^3}{x+1}\,,\]
then there is a place $Q$ of $F_{i+1}$ such that
\[d(Q|P)=5\cdot 3^{i-1}-1\geq \frac{5}{6}[F_i:F_0]\,,\]
where $P$ is the place of $F_i$ lying under $Q$. In this case \eqref{seriesdiffexp} holds because $[F_{i+1}:F_i]=3$ for all $i\geq 0$.

In other cases  the infiniteness of the genus of a tower $\m{F}=(F_0, F_1, \ldots)$ over $K$  can be proved by showing that for some subtower $\m{F}'=(F_{s_1}, F_{s_2}, \ldots)$ of $\m{F}$ sufficiently many places of $F_{s_i}$ are ramified in  $F_{s_{i+1}}$ in the  sense that the number $r_i=\#(R_i)$ where
\[R_i=\{P\in\P(F_{s_i})\,:\,\text{$P$ is ramified in $F_{s_{i+1}}$}\}\,.\] satisfies the estimate
\[r_i\geq c_i[F_{s_{i+1}}:F_0]\,,\]
where $c_i>0$ for $i\geq 1$ and the series
$\sum_{i=1}^{\infty}c_i$
is divergent. It is easily seen that the divisor of $F_{s_i}$
\[B_i=\sum_{P\in R_i}P\,,\]
satisfies the conditions \eqref{thm3.2-a} and \eqref{thm3.2-b} of Remark \ref{remarkdivisor}
with $c_1(i)= c_i[F_{s_{i+1}}:F_{s_i}]$ and $c_2(i)=[F_{s_{i+1}}:F_{s_i}]^{-1}$.

An example of the situation just described  is given in \cite{BeelenGS05}. The authors consider a recursive tower $\m{F}=(F_0, F_1, \ldots)$ and its dual tower $\m{G}=(G_0, G_1, \ldots)$ with $F_0=G_0$. They show that if $\m{G}$ is a tower and there is a place $P$ of $F_0$ such that
\[\sup_{i\geq 1}\{e(Q|P)\,:\,Q\in\P(F_i)\}\neq \sup_{i\geq 1}\{e(Q|P)\,:\,Q\in\P(G_i)\}\,,\]
then there exist a subtower  $\m{F}'=(F_{s_1}, F_{s_2}, \ldots)$ of $\m{F}$ and  positive integers $k$ and $n$ such that $s_{i+1}=s_i+k$ and at least $m^{-k}[F_{s_{i+1}}:F_0]$ places of $F_{s_i}$ are ramified in  $F_{s_{i+1}}$ where $m$ is the degree of each extension $F_{i+1}/F_i$.

With additional hypothesis, in \cite{CTcubic}, the following result was proved. 
\begin{proposition}\label{propmajwulf} Let $\m{F}=(F_0, F_1, \ldots)$ be  a tower of function fields over $K$. Suppose that either each extension $F_{i+1}/F_i$ is Galois or that there exists a constant $M$ such that $[F_{i+1}:F_i]\leq M$ for $i\geq 0$. In order to have infinite genus it suffices to find, for infinitely many indices $i\geq 1$, a place $P_i$ of $F_0$ unramified in $F_i$ and such that each place of $F_i$ lying above $P_i$ is ramified in $F_{i+1}$.
\end{proposition}

Recall that a tower $\m{F}=(F_0, F_1, \ldots)$ of function fields over a perfect field $K$ of positive characteristic is called \emph{tame} if the ramification index $e(Q|P)$ of any place $Q$ of $F_{i+1}$ lying above a place $P$ of $F_i$ is relatively prime to the characteristic of $K$ for all $i\geq 0$. In the case of a tame tower the Dedekind's different theorem implies that $d(Q|P)\leq e(Q|P)$. Thus if $K$ is algebraically closed every place is of degree one, all the inertia degrees are one and then we have
\begin{align*}
\deg \diff(F_{i+1}/F_i) &=\sum_{P\in R_i}\underset{Q|P}{\sum_{Q\in\P(F_{i+1})}}d(Q|P)\deg(Q)\\
& \leq \sum_{P\in R_i}\underset{Q|P}{\sum_{Q\in\P(F_{i+1})}}e(Q|P)\leq[F_{i+1}:F_i]r_i\,,
\end{align*}
where $R_i=\{P\in\P(F_i)\,:\,P\text{ is ramified in } F_{i+1}\},$ and $r_i=\#(R_i)$. Therefore we clearly have the following result which states that the genus of some particular class of tame towers $\m{F}=(F_0, F_1, \ldots)$ depends only on the number of ramified places from $\bar{F}_i$ to $\bar{F}_{i+1}$ where $\bar{F_i}=F_i\cdot K'$ and $K'$ is an algebraic closure of $K$.

\begin{proposition}\label{genustame}
  Let $\m{F}=(F_0, F_1, \ldots)$ be a tame tower of function fields over a perfect field $K$ of positive characteristic satisfying the conditions of Proposition \ref{propmajwulf}.  Let
\[R_i=\{P\in\P(\bar{F}_i)\,:\,P \text{ is ramified in } \bar{F}_{i+1}\}\,,\]
where $\bar{F_i}=F_i\cdot K'$ with $K'$ an algebraic closure of $K$  and let $r_i=\#(R_i)$. Then
\begin{equation}\label{ineqgenustame}
g(F_0)-1+\frac{1}{2M}\sum_{i=0}^\infty\frac{r_i}{[F_i:F_0]}\leq \gamma(\m{F})\leq g(F_0)-1+\frac{1}{2}\sum_{i=0}^\infty\frac{r_i}{[F_i:F_0]}\,.
\end{equation}
\end{proposition}

Now we  prove the following partial converse of the situation described in Remark \ref{remarkdivisor}.
\begin{proposition}\label{reciprocalremark3.2}
  Let $\m{F}=(F_0, F_1, \ldots)$ be a tame tower of function fields over $K$ of positive characteristic such that $\gamma(\m{F})=\infty$. Suppose that $\m{F}$ satisfies the conditions of Proposition \ref{propmajwulf}. Then for each $i\geq 1$ there exist a divisor  $B_i\in \m{D}(F_i)$ and functions $c_1(t)$ and $c_2(t)$ such that conditions \eqref{thm3.2-a} and \eqref{thm3.2-b} of Remark \ref{remarkdivisor} hold and the series $\sum_{i=1}^{\infty}c_1(i)c_2(i)$ is divergent.
\end{proposition}
\begin{proof}
We may assume that $K$ is algebraically closed. Let $R_i$ be as in Proposition \ref{genustame} and let $r_i=\#(R_i)$. Since $\gamma(\m{F})=\infty$ then by \eqref{ineqgenustame} we have that the series \[\sum\limits_{i=0}\limits^\infty\frac{r_i}{[F_i:F_0]}\,,\] is divergent. Let $$B_i=\sum_{P\in R_i}P.$$
We have $$\deg B_i=\sum_{P\in R_i}\deg P \geq r_i=c_1(i)[F_i:F_0]\,,$$
  where $c_1(i)= r_i[F_i:F_0]^{-1}$.

If each extension $F_{i+1}/F_i$ is Galois, then condition~\eqref{thm3.2-b} of Remark \ref{remarkdivisor} holds with $c_2(i)=1/2$ because we have
  $$\sum_{P\in \supp{B_i}}\underset{Q|P}{\sum_{Q\in\P(F_{i+1})}}d(Q|P)\deg Q\geq \frac{1}{2} [F_{i+1}:F_i]\deg B_i.$$
If $[F_{i+1}:F_i]\leq M$ for some positive constant $M$ and for all $i\geq 0$, then condition~\eqref{thm3.2-b} of Remark \ref{remarkdivisor} holds with $c_2(i)=1/M$ because we obtain
  $$\sum_{P\in \supp{B_i}}\underset{Q|P}{\sum_{Q\in\P(F_{i+1})}}d(Q|P)\deg Q \geq \deg B_{i}\geq \frac{1}{M}[F_{i+1}:F_i]\deg B_{i}.$$
Hence, since $M\geq 2$, we have
\[\sum_{i=1}^{\infty} c_1(i)c_2(i)\geq \frac{1}{M}\sum_{i=1}^{\infty}\frac{r_i}{[F_i:F_0]}=\infty\,,\]
as desired.
\end{proof}

\section{Examples}\label{examples}
In some examples of this section we shall use the following convention: a place defined by a monic and irreducible polynomial $f\in K[T]$ in a rational function field $K(x)$ will be denoted by $P_{f(x)}$.
\begin{example}
 Let $\m{F}=(F_0, F_1, \ldots)$ be a non-skew recursive tower of function fields over a perfect field $K$ of characteristic $p$ defined by the polynomial
\[H=(T^{p+1}+T)f(S)-S^{p+1}\in K[S,T]\,,\]
where $f\in K[S]$ is a polynomial of degree $p+1-r$ with $\gcd(p+1,r)=1$ and such that $f(0)\neq 0$. 
The basic function field is $K(x,y)$ with
\[y^{p+1}+y=\frac{x^{p+1}}{f(x)}\,.\]


Let $Q$ be a zero of $y$ in $K(x,y)$. Then $Q$ lies above $P_y$ and, since $\nu_Q(y)>0$, we have
\[\nu_Q(y^{p+1}+y)=\nu_Q(y)+\nu_Q(y^p+1)=\nu_Q(y)\,,\]
and if $P=Q\cap K(x)$ then
\[0<\nu_Q(y^{p+1}+y)=e(Q|P)((p+1)\nu_P(x)-\nu_P(f(x)))\,,\]
Therefore $\nu_P(x)>0$ (the inequality $\deg f<p+1$ is used to rule out the possibility $\nu_P(x)<0$) and then $P=P_x$ which implies that $Q$ lies over $P_x$. Now let $Q'$ be a zero of $y^p+1$ in $K(x,y)$ and let $R=Q'\cap K(x)$. Then $\nu_{Q'}(y)=0$ and
\[0<\nu_{Q'}(y^p+1)=\nu_{Q'}(y^{p+1}+y) =e(Q'|R)((p+1)\nu_R(x)-\nu_R(f(x)))\,.\]
This implies that $\nu_R(x)>0$  showing that $R=P_x$ so that $Q'$ lies above $P_x$. Then $\nu_{P_x}(f(x))=0$ and using that we are in characteristic $p$ and that $\nu_{Q'}(y)=0$ we have
\[p\nu_{Q'}(y+1)=\nu_{Q'}(y^{p+1}+y)=(p+1)e(Q'|P_x)\nu_{P_x}(x)=(p+1)e(Q'|P_x)\,.\]
Since $[K(x,y):K(x)]=p+1$, the above equality forces to have $e(Q'|P_x)=p$ and then $e(Q|P_x)=1$.

On the other hand, since $Q$ is a place of $K(x,y)$ lying above $P_y$, we have
\[(p+1)\nu_Q(x)-\nu_Q(f(x))=e(Q|P_y)(\nu_{P_y}(y)+\nu_{P_y}(y^p+1))=e(Q|P_y)\geq 1\,,\]
and this implies that $\nu_Q(x)>0$. Hence $\nu_Q(f(x))=0$ and then
\[p+1\leq (p+1)\nu_Q(x)=e(Q|P_y)\leq p+1\,,\]
so that $e(Q|P_y)=p+1$. Summarizing we have that if $Q$ is a zero of $y$ in $K(x,y)$ then $Q$ lies above $P_y$ in $K(y)$ and above $P_x$ in $K(x)$, $e(Q|P_y)=p+1$ (hence $P_y$ is tamely and totally ramified in $K(x,y)$ and $\gcd(e(Q|P_y),p)=1$), $e(Q|P_x)=1$ (hence $\gcd(e(Q|P_y),e(Q|P_x))=1$) and if $Q'$ is  a zero of $y^p+1$ in $K(x,y)$ then  $Q'$ lies above $P_x$ in $K(x)$ and $e(Q'|P_x)=p$ so that $Q'|P_x$ is wildly ramified and $\gcd(e(Q'|P_x),e(Q|P_y))=1$.

We have thus that $\m{F}$ is, in fact, an asymptotically bad tower over $K$. As we mentioned in the previous section, in \cite{GS96}  a particular case of this example was presented with $p=2$ and $f(S)=S+1$.

\end{example}

\begin{example}\label{ejemploBGS04}
We shall show now that the result proved in \cite[Theorem 2.1]{BGS04} on asymptotically bad towers of Artin-Schreier type can be deduced from Corollary \ref{coroTeoclasificacion1}. Let $\m{F}=(F_0, F_1, \ldots)$ be an Artin-Schreier tower of function fields over a perfect field $K$ of characteristic $p>0$ defined by the polynomial
\[H=(T^p-T)b_2(S)-b_1(S)\in K[S,T]\,,\]
where $b_1$, $b_2\in K[S]$ are coprime polynomials of degree $\deg{b_1}=p$ and $\deg{b_2}=r<p$ with $\gcd(r,p)=1$. Let $b(S)=b_1(S)/b_2(S)$ and let $\{x_i\}_{i=0}^{\infty}$ be a sequence of transcendental elements over $K$ such that
\[F_0=K(x_0)\quad\text{and}\quad F_{i+1}=F_i(x_{i+1})\,,\]
with $x_{i+1}^p-x_{i+1}=b(x_i)$ for $i\geq 0$. Suppose that the set $$L_0=\{P\in Supp((b(x_0))_\infty^0): \gcd(v_P(b(x_0)),p)=1\}\subset \P(F_0)\,,$$ is non-empty and let $L_j=\{P'\in \P(F_j):P'|P\text{ for some }P\in L_{j-1}\}$, for $j\geq 1$. Finally suppose that there is a constant $C>0$ such that for infinitely many indices $0\leq r_1<r_2<\cdots$ we have that $$\deg{B_{r_j}}\geq C\cdot \deg(b(x_{r_j}))_\infty^{r_j}\,,$$ where $$B_{r_j}:=\sum_{P\in L_{r_j}} -v_{P}(b(x_{r_j}))P\quad\in \m{D}(F_{r_j})\,.$$ We claim that $\m{F}$ is, in fact, an asymptotically bad tower. To see this as a consequence of Corollary \ref{coroTeoclasificacion1}, first we prove that $$L_j\subset\{P\in \P(F_j): v_P(b(x_{j}))<0 \quad\text{and}\quad\gcd(v_P(b(x_{j})),p)=1\},$$ for all $j\geq 0$.
By definition, we know that $$L_0=\{P\in Supp((b(x_0))_\infty^0): \gcd(v_P(b(x_0)),p)=1\}\subset \P(F_0).$$
Let $P$ be a place in $L_0$, $Q$ a place of $F_1$ lying above $P$ and $z_j=b(x_{j-1})$ for $j\geq 1$. Since $P$ is a pole of $b(x_0)$ we have  $$v_Q(x_1^p-x_1)=e(Q|P)v_P(b_1(x_0)/b_2(x_0))=e(Q|P)v_P(z_1)<0\,,$$
and this implies that $v_Q(x_1)<0$. Hence
$$v_Q(z_2)=v_Q(b(x_1))=v_Q(b_1(x_1))-v_Q(b_2(x_1))=(p-r)v_Q(x_1)<0\,,$$
and $$\gcd(v_Q(z_2),p)=\gcd((p-r)v_Q(x_1),p)=1.$$

Now we proceed by induction. Assume that $$L_j\subset\{P\in \P(F_j): v_P(z_{j+1})<0\quad \text{and }\quad\gcd(v_P(z_{j+1}),p)=1\}\,,$$ for $j\leq k$. Let $Q$ be place in $L_k$ and let $R$ be a place of $F_{k+1}$ lying above $Q$. Then $$v_R(x_{k+1}^p-x_{k+1})=e(R|Q)v_Q(b(x_k))=e(R|Q)v_Q(z_{k+1})<0\,,$$ and again we have that $v_R(x_{k+1})<0.$ Hence $$v_Q(z_{k+2})=(p-r)v_Q(x_{k+1})<0\,,$$ and $$\gcd(v_Q(z_{k+2}),p)=1\,.$$

We claim now that if $P\in L_{j}$ then for some $u\in F_j$ we have that $$v_P(z_{j+1}-(u^p-u))=-m<0\,,\qquad\text{with}\qquad m\not\equiv 0 \mod p\,.$$ To prove the claim we argue as follows: if $u\in F_j$ is such that $v_P(u^p-u)=v_P(z_{j+1})$ then $0>v_P(u^p-u)$ so that $v_P(u)<0$. Hence $v_P(z_{j+1})=v_P(u^p-u)=p\,v_P(u)$ which contradicts the assumption $\gcd(v_P(z_{j+1}),p)=1$. Therefore, for all $u\in F_j$ we have $v_P(z_{j+1})\neq v_P(u^p-u)$. By the so called strict triangular inequality (see \cite[Lemma 1.1.11]{Stichbook09}) we get $$v_P(z_{j+1}-(u^p-u))\leq v_P(z_{j+1})<0\,,$$ for any $u\in F_j$.  From this inequality and Lemma 3.7.7 in \cite{Stichbook09} we have that there exists $u\in F_j$ such that $$v_P(z_{j+1}-(u^p-u))=-m<0\qquad\text{ with }\qquad m\not\equiv 0 \mod p\,,$$ as claimed.

Finally, we claim that if $P\in L_{r_j}$, $Q \in L_{r_{j+1}}$ and $Q|P$ then
 \begin{equation}\label{ecunueva}
d(Q|P)\geq \frac{1}{2}(-v_{P}(b(x_{r_j})))[F_{r_{j+1}}:F_{r_j}]\,,
\end{equation} for all $j\geq 1$.
To see this note that if $L_{r_{j+1}}/L_{r_j}$ is an Artin-Schreier extension, then
 \begin{align*}
d(Q|P)&\geq (-v_{P}(b(x_{r_j}))+1)([F_{r_{j+1}}:F_{r_j}]-1)\,,\\
&\geq \frac{1}{2}(-v_{P}(b(x_{r_j})))[F_{r_{j+1}}:F_{r_j}]\,,\\
\end{align*}
by the general theory of Artin-Schreier extensions (see \cite[Theorem 3.7.8, (c)]{Stichbook09}).
If not we have that  $L_{r_{j}+2}/L_{r_{j}+1}$ and  $L_{r_{j}+1}/L_{r_j}$ are Artin-Schreier extensions and
if $P^1=Q\cap F_{r_{j}+1}$ and $P^2=Q\cap F_{r_{j}+2}$ then we have that
 \begin{align*}
d(P^1|P)&\geq \frac{1}{2}(-v_{P}(b(x_{r_j})))[F_{r_{j}+1}:F_{r_j}]\,,
\end{align*}
and also that $$e(P^2|P^1)=[F_{r_{j}+2}:F_{r_{j}+1}].$$
Now by the transitivity of the different exponent we get
  \begin{align*}
d(P^2|P)&\geq e(P^2|P^1)d(P^1|P)\\
&\geq [F_{r_{j}+2}:F_{r_{j}+1}]\frac{1}{2}(-v_{P}(b(x_{r_j})))[F_{r_{j}+1}:F_{r_j}]\\
&=\frac{1}{2}(-v_{P}(b(x_{r_j})))[F_{r_{j}+2}:F_{r_j}]\,.
\end{align*}
If $L_{r_{j+1}}=L_{r_{j}+2}$ we are done. If not, since $L_{r_{j}+3}/L_{r_{j}+2}$ is an Artin-Schreier extension, then we have that
 $$e(P^3|P^2)=[F_{r_{j}+3}:F_{r_{j}+2}]\,,$$ where $P^3=Q\cap F_{r_{j}+3}$.
 Therefore,
 \begin{align*}
d(P^3|P)&\geq e(P^2|P^1)d(P^1|P)\\
&=\frac{1}{2}(-v_{P}(b(x_{r_j})))[F_{r_{j}+3}:F_{r_j}]\,.
\end{align*}
Again if  $L_{r_{j+1}}=L_{r_{j}+3}$ we are done. If not, by continuing with an inductive argument we see that que claim holds.

Therefore using \eqref{ecunueva} we have
\begin{align*}
  \sum_{P\in supp(B_{r_j})}\sum_{P'|P}d(P'|P)\deg P'
  &\geq \frac{1}{2}[F_{r_{j+1}}:F_{r_j}] \sum_{P\in supp(B_{r_j})}-v_{P}(b(x_{r_j})) \deg P\\
  &=\frac{1}{2}[F_{r_{j+1}}:F_{r_j}] \deg{B_{r_j}}\,,
 \end{align*} and we conclude that $\gamma(\m{F})=\infty$  from Corollary \ref{coroTeoclasificacion1} because in this case $c_1(j)=C$ and $c_2(j)=1/2$ for all $j\geq 0$.
 \end{example}

\begin{example}
Let $\m{F} = (F_0, F_1, \ldots)$ be a non-skew recursive tower of function fields over $K$ defined by a polynomial $H(S,T)\in K[S,T]$. Following \cite{BeelenGS05}, we define its dual tower $\m{G} =  (G_0, G_1, \ldots)$ as the recursive tower defined by $H(T,S)$. We identify the rational function
field $F_0 = K(x_0)$ with $G_0 = K(y_0)$ by setting $x_0 = y_0$. Then we
have that $F_0=G_0$ and
\begin{align*}
  F_n = K(x_0,\ldots, x_n) &\quad \text{with } H(x_i, x_{i+1}) = 0, \text{ and }\\
  G_n = K(y_0,\ldots, y_n) &\quad \text{with } H(y_{i+1}, y_i) = 0
\end{align*} for all $n \geq 1$ and $0 \leq i \leq n - 1$. Note that the function fields $F_i$ and $G_i$ are $K$-isomorphic for all $i\geq 0$.

For $P\in \P(F_0)= \P(G_0)$ we define the set
$$\epsilon(P,\m{F}):=\sup_{n\geq 1}\{e(Q_n|P):Q_n \in \P(F_n)\text{ and }Q_n|P\}$$   We will show that if
 $$\epsilon(P,\m{F})\neq \epsilon(P,\m{G})\,,$$
then $\gamma(F)=\infty$ so that $\m{F}$ is actually a tower which is asymptotically bad.
 In order to prove this, we follow the proof given in \cite{BeelenGS05} by considering  $\m{F}$ as a tower over an algebraic closure $K'$ of $K$ (recall that the genus of a tower and the ramification indices do not change in constant field extensions). Hence any place of each $F_i$ is of degree one. We also have that
$$[F_{n+1}:F_n] = [G_{n+1}:G_n] = m\,,$$
for all $n\geq 1$ where $1<m=\deg_SH = \deg_T H$.

Without loss of generality we can assume that $\epsilon(P,\m{F}) > \epsilon(P, \m{G})$. Then $e_1=
\epsilon(P,\m{G})$ is a positive integer. By definition of $\epsilon(P, \m{G})$ there exist a positive integer $n$ and a place $Q_1 \in \P(G_n)$ such that
\begin{enumerate}[(i)]
\item $e(Q_1|P) = e_1$.
\item $Q_1$ splits completely in $G_l/G_n$ for all $l\geq n$.
\end{enumerate}

Since  $\epsilon(P,\m{F}) > \epsilon(P, \m{G})$ there exists a positive integer $k$ such that there is a place $Q_2 \in \P(F_k)$ lying above $P$ with $$e_2= e(Q_2|P) > e_1\,.$$
Let $l\geq n$ and let $H_l:=F_k\cdot G_l$ (resp. $H_n:= F_k\cdot G_n$) be the composite
field of $F_k$ with $G_l$ (resp. with $G_n$). Consider a place $R_1 \in \P(G_l)$ lying above
the place $Q_1$ (see Figure \ref{figu8}).

 \begin{figure}[h!t]
 \begin{center}
  \begin{tikzpicture}[scale=1]
   \draw[line width=0.5 pt](0,6)--(4,0)--(7,4.5)--(6,6)--(3,1.5);
    \draw[line width=0.5 pt](6,3)--(5,4.5);
     \draw[line width=0.5 pt](8,1.5)--(9,0)--(12,4.5)--(11,6)--(8,1.5);
    \draw[line width=0.5 pt](11,3)--(10,4.5);
        \draw[white, fill=white](4,0) circle (0.3 cm);
            \draw[white, fill=white](3,1.5) circle (0.3 cm);
                \draw[white, fill=white](1,4.5) circle (0.3 cm);
                    \draw[white, fill=white](0,6) circle (0.3 cm);
                        \draw[white, fill=white](5,4.5) circle (0.3 cm);
                            \draw[white, fill=white](6,3) circle (0.3 cm);
                              \draw[white, fill=white](7,4.5) circle (0.3 cm);
                              \draw[white, fill=white](6,6) circle (0.3 cm);
   \draw[white, fill=white](9,0) circle (0.3 cm);
            \draw[white, fill=white](8,1.5) circle (0.3 cm);
                        \draw[white, fill=white](10,4.5) circle (0.3 cm);
                            \draw[white, fill=white](11,3) circle (0.3 cm);
                              \draw[white, fill=white](12,4.5) circle (0.3 cm);
                              \draw[white, fill=white](11,6) circle (0.3 cm);
                                \node at(4,0){$F_0=G_0$};
\node at(3,1.5){$F_k$};
 \node at(1,4.5){$F_l$};
  \node at(0,6){$F_{l+k}$};
 \node at(6,3){$G_n$};
 \node at(7,4.5){$G_l$};
 \node at(5,4.5){$H_n=F_k\cdot G_n$};
 \node at(6,6){$H_l=F_k\cdot G_l$};
   \node at(9,0){$P$};
\node at(8,1.5){$Q_2$};
 \node at(11,3){$Q_1$};
 \node at(12,4.5){$R_1$};
 \node at(10,4.5){$R_2$};
 \node at(11,6){$S_1$};
  \node at(5.5,1.5){\scriptsize{$m^n$}};
  \node at(7,3.7){\scriptsize{$m^{l-n}$}};
   \node at(6.9,5.4){\scriptsize{$m^{k}$}};
   \node at(10.5,1.5){\scriptsize{$e_1$}};
  \node at(11.7,3.7){\scriptsize{$1$}};
   \node at(12,5.3){\scriptsize{$e>1$}};
   \node at(8.3,.7){\scriptsize{$e_2$}};
   \node at(10.9,4){\scriptsize{$e>1$}};
  \end{tikzpicture}
  \caption{Ramification of $P$ in $\m{F}$ and $\m{G}$.}\label{figu8}
\end{center}\end{figure}
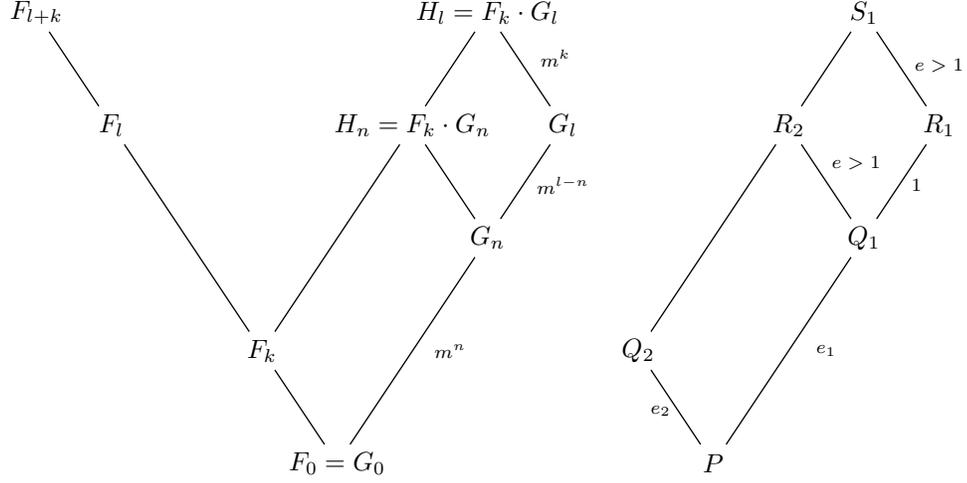

 From \cite[Lemma 2.1]{MaWu05} there is a place $R_2\in \P(H_n)$ lying above $Q_1$ and $Q_2$. Since $e_2>e_1$ we have $e(R_2|Q_1)>1$. Again by \cite[Lemma 2.1]{MaWu05} there exists a place $S_1\in \P(H_l)$ lying above $R_1$ and $R_2$. It follows that $e(S_1|R_1)=e(R_2|Q_1)>1$. Since $$\#\{R_1 \in \P(G_l):R_1|Q_1\}=[G_l:G_n]=m^{l-n}$$
we conclude that there are, at least, $m^{l-n}$ places of $G_l$ which ramify in $H_l$. On the other hand since $H_l=K'(x_k,\ldots,x_1,x_0=y_0,y_1\ldots,y_l)$ we see that $H_l$ is isomorphic to $F_{l+k}$ by means of the map $\sigma$ sending $x_i$ to $y_{l+i}$ for $0\leq i\leq k$ and $y_j$ to $x_{l-j}$ for $1\leq j\leq l$. Since  this map restricted to $G_l=K'(y_0,\ldots,y_l)$ gives $F_l=K'(x_0,\ldots,x_l)$ we have that if a place $P$ of $G_l$ ramifies in $H_l$ then $\sigma(P)$ is a place of $F_l$ which is ramified in $F_{l+k}$. Therefore  at least $m^{l-n}$ places of $F_l$ ramify in $F_{l+k}$ for $l\geq n$.

Now let us consider the sequence $\{r_j\}_{j\geq 0}$ where $r_j=n+jk$. We know that in each extension $F_{r_j}$ there are at least $$m^{r_j-n}=
m^{n+jk-n}=m^{jk}$$ places, $P_1, P_2, \ldots, P_{m^{jk}}$, which ramify in the extension $F_{r_{j+1}}$. Let
$$B_{r_j}=\sum_{i=1}^{m^{jk}}P_i.$$
Then
$$\deg{B_{r_j}}=\sum_{i=1}^{m^{jk}}\deg P_i=m^{jk}\geq [F_{r_{j}}:F_0]\,,$$
which is \eqref{thm3.2-a} of Remark~\ref{remarkdivisor} with $c_1(j)=1$.
Also \begin{align*}
  \sum_{P\in \supp{B_{r_j}}}\sum_{P'|P} d(P'|P)\deg{P'}&\geq \sum_{i=1}^{m^{jk}}(e(P'|P)-1)\\
  &\geq m^{jk}\\
  &=\frac{1}{m^{k}}[F_{r_{j+1}}:F_{r_j}]\deg{B_{r_j}}\,,
\end{align*}
so that $c_2(j)=m^{-k}$ in {(b)} of Remark~\ref{remarkdivisor} and then $\m{F}$, as a tower over $K'$, has infinite genus by Remark~\ref{remarkdivisor}. Therefore  $\m{F}$ has infinite genus as a tower over $K$ and then $\m{F}$ is, in fact, an asymptotically bad tower over $K$.
\end{example}

\bibliographystyle{amsplain}

\end{document}